\documentclass[11pt]{article}
\usepackage[arrow, matrix]{xy}
\usepackage{amsmath}
\usepackage{cancel}
\usepackage{amssymb}
\usepackage{amsthm}
\usepackage[mathscr]{eucal}
\usepackage{lineno}

\input{epsf}
\theoremstyle{definition}
\newtheorem{definition}{Definition}
\newtheorem{theorem}[definition]{Theorem}

\newtheorem{lemma}[definition]{Lemma}

\theoremstyle{remark}
\newtheorem{remark}[definition]{Remark}

\newcounter{enumctr}

\newcommand{\N}{\mathbb{N}}
\newcommand{\R}{\mathbb{R}}

\renewcommand{\phi}{\varphi}

\begin{document}

\title{\vspace*{-30mm}
Semi-dynamical systems generated by  autonomous Caputo fractional  differential equations}
\author{Thai  Son Doan\footnote{Email: dtson@math.ac.vn, Institute of Mathematics, Vietnam Academy of Science and Technology, 18 Hoang Quoc Viet, Hanoi, Viet Nam.}\; and Peter E. Kloeden\footnote{Email: kloeden@na.uni-tuebingen.de, Mathematisches Institut, Universit\"at T\"ubingen, 72074 T\"ubingen, Germany}.}

\date{}

\maketitle
\begin{abstract}
	An autonomous    Caputo fractional   differential equation  of order $\alpha\in(0,1)$  in $\mathbb{R}^d$ whose vector field satisfies a global  Lipschitz condition  is shown to  generate a semi-dynamical system in the  function space  $\mathfrak{C}$   of continuous functions $f:\R^+\rightarrow \R^d$   with the topology uniform convergence on compact subsets.
	This contrasts with a recent result of Cong \& Tuan \cite{cong}, which showed that such equations do not, in general, generate a dynamical system on the space $\mathbb{R}^d$.
\end{abstract}
\emph{2010 Mathematics Subject Classfication.} Primary:34A08, 34A10, 34B10; Secondary:  34C11, 34C35.


\section{Introduction}

The asymptotic behaviour of Caputo fractional   differential equations  in $\mathbb{R}^d$  has attracted much attention in the literature in recent years. It has often been asked if such equations generate an   autonomous (or nonautonomous,  if appropriate) dynamical system, since that would allow the theory of attractors to applied to them.  In 2017 Cong \& Tuan \cite{cong} showed that  such equations do not  generate a dynamical system on $\mathbb{R}^d$, except in special cases.

In this note, we observe that an ``autonomous"  Caputo fractional   differential equation (Caputo FDE), i.e., with a time independent vector field,   is formulated as an integral equation similar to a Volterra integral equation, but with an integrably singular rather than continuous kernel. This opens the door to Miller and Sell's formulation of  Volterra integral equations as autonomous semi-dynamical systems, see Miller \& Sell \cite{MS1} and Sell \cite[Chapter XI]{sell}, and enables us to  determine an autonomous semi-dynamical system  representation of   autonomous   Caputo FDEs  on the  function space  $\mathfrak{C}$   of continuous functions $f:\R^+\rightarrow \R^d$   with the topology uniform convergence on compact subsets.%

Consider an  autonomous   Caputo fractional   differential equation  of order $\alpha\in(0,1)$ in $\mathbb{R}^d$  of the following form
\begin{equation}\label{acfde}
^{\!C}D^{\alpha}_{0+} x(t)= g(x(t))
\end{equation}
where $g : \R^d\rightarrow \R^{d}$  is  globally Lipschitz  continuous.  We represent the solution  of the  Caputo FDE \eqref{acfde} with initial condition $x(0)$ $=$ $x_0$
by  the integral equation
\begin{equation}\label{AIE}
x(t)=x_0+\frac{1}{\Gamma(\alpha)}
\int_0^t (t-s)^{\alpha -1} g(x(s))ds,
\end{equation}
where $\Gamma(\alpha):=\int_0^\infty t^{\alpha-1}\exp{(-t)}dt$ is the Gamma function.

Define
$$
a(t,s) :=  \frac{1}{\Gamma(\alpha)}  (t-s)^{\alpha -1},  \qquad 0\leq s < t.
$$
Then   the integral  equation \eqref{AIE} is a special case of the (singular) Volterra integral equation
\begin{equation}\label{SVIE}
x(t)=   f(t)+   \int_0^t a(t,s)  g(x(s)) ds
\end{equation}
where   $f:\R^+\rightarrow \R^d$ is  a continuous function. In the case of an   Caputo FDE
\eqref{acfde} $f(t)$ $\equiv$ $f(0)$ $=$ $x_0$.

As preparation, we first establish the existence and uniqueness of solutions of the integral  equation \eqref{AIE} on any bounded time  interval  $0,T]$ for each $f$ $\in$
$\mathfrak{C}$  and then show their continuity in the initial data. For this we use the contraction mapping principle on the space $C([0,T],\rightarrow \R^d)$ with a norm weighted
by an  appropriate Mittag-Leffler function. The results assume that the vector field $g$ satisfies a global Lipschitz condition, but as in Sell \cite{sell} we establish the semi-group
property for a larger class of admissible vector fields,  which are assumed  to satisfy these preparatory results globally in  time.

The extension to ``nonautonomous"    Caputo fractional   differential equations and skew-product flows is  sketched in the final section.

\section{Preliminaries}\label{Section2}

The existence (local)  and uniqueness and continuity in $f$ $\in$ $C([0,T],\mathbb{R}^d)$   of solutions of  \eqref{SVIE} are given in   Miller \cite{miller} and Sell \cite{sell}   provided that $a(t,s)$ is continuous at $s$ $=$  $t$.  In our case  $a(t,s)$  is integrably singular, but   we can adapt  the proof in  Doan et al \cite{doan}, which is for It\^o stochastic versions of Caputo  FDE to give the  global  existence and uniqueness of solutions; see also \cite{cong}.

\subsection{Global existence and uniqueness solutions}  The global existence and uniqueness solutions of \eqref{acfde} and of the more general  integral equation will be established when the vector  field $g$ satisfies   the global Lipschitz condition:
\begin{itemize}
	\item [(H1)] There exists $L>0$ such that for all $x,y\in\R^d$, $t\in [0,\infty)$
	\begin{equation*}\label{L_cond}
	\|g(x)-g(y)\|\leq L\|x-y\|.
	\end{equation*}
\end{itemize}

The proof follows by a contraction mapping argument, which gives only local existence if the usual supremum norm on continuous functions is used. Unlike ODES, these local solutions cannot be
patched together to  provide a global solution  for Caputo  FDE.  The proof in Doan et al \cite{doan} for stochastic DEs  can be adapted to  this case  using a Banach space with a  suitable Bielecki
weighted norm
$$
\|x\|_{\gamma}:=\sup_{t\in [0,T]}
\frac{ \|x(t)\|}{E_{\alpha}(\gamma t^{\alpha})}\qquad \hbox{for all }  x \in C([0,T], \mathbb{R}^d).
$$
where  $\gamma>0$ is a suitable constant and the weight function is the Mittag-Leffler function $E_{\alpha}(\cdot)$  defined as follows:
\[
E_{\alpha}(t):= \sum_{k=0}^{\infty}\frac{t^k}{\Gamma(\alpha k+1)}\quad \hbox{for all}\; t\in\R.
\]

\begin{theorem}\label{Theorem1}
	Assume  that the vector field $g$ satisfies the global Lipschitz condition  in Assumption H1. Then for any $T$ $>$ $0$ and for each $f$ $\in$ $C([0,T],\mathbb{R}^d)$ the integral equation \eqref{SVIE} has a  unique
	solution $x(t,f)$ on the interval $[0,T]$.
\end{theorem}
\begin{proof}
	Since the proof is standard we just show the contraction property and how  the weighted norm is used.  Let $x, y, f\in$ $C([0,T],\mathbb{R}^d)$ with
	$$
	(\mathfrak{T}x)(t)=   f(t)+   \int_0^t a(t,s)  g(x(s))\;ds, \quad (\mathfrak{T}y)(t)=   f(t)+   \int_0^t a(t,s)  g(y(s))ds,
	$$
	for each $t \in [0,T]$. Then
	\begin{eqnarray*}
		\left\|(\mathfrak{T}x)(t)-(\mathfrak{T}y)(t)\right\|  & \leq &\int_0^t a(t,s) \left\|g(x(s))-  g(y(s))\right\| ds
		\\[1.4ex]
		& \leq & L \int_0^t a(t,s) \left\|x(s) -   y(s)\right\| ds .
	\end{eqnarray*}
	By definition of $\|\cdot\|_{\gamma}$,
	\begin{equation}\label{Eq1_Son}
	\left\|(\mathfrak{T}x)(t)-(\mathfrak{T}y)(t)\right\|
	\leq
	L\int_0^t a(t,s) E_{\alpha}(\gamma s^{\alpha})\;ds\; \|x-y\|_{\gamma}.
	\end{equation}
	Since $E_{\alpha}(\gamma t^{\alpha})$ is a solution of the linear fractional differential equation $^{\!C}D^{\alpha}_{0+} x(t)= \gamma x(t) $ it follows that
	\[
	E_{\alpha}(\gamma t^{\alpha})=1+\gamma\int_0^t a(t,s) E_{\alpha}(\gamma s^{\alpha})ds,
	\]
	which together with \eqref{Eq1_Son} implies that
	\[
	\frac{\left\|(\mathfrak{T}x)(t)-(\mathfrak{T}y)(t)\right\|}{E(\gamma t^{\alpha})}
	\leq \frac{L}{\gamma}\|x-y\|_{\gamma}.
	\]
	Hence, by choosing $\gamma> L$ the operator $\mathfrak T$ is a contraction on $(C([0,T],\R^d),\|\cdot\|_{\gamma})$ and its unique fixed point gives the unique solution of \eqref{acfde}. The proof is complete.
\end{proof}

\subsection{Continuous dependence of the solution on the   input function}   We can also show the continuous dependence of solutions on the input function $f$, but we do
not need the weighted norm for this. Instead, we will use  the  following version of Gronwall's lemma  from Diethelm \cite[Lemma 6.19]{diethelm}.
\begin{lemma}\label{fGron}
	Let $\alpha$, $\mu$, $\nu$, $T$ $\in$ $\mathbb{R}^+$ and let $\Delta$ $:$ $[0,T]$ $\rightarrow$ $\mathbb{R}$ be a continuous function satisfying the inequality
	$$
	|\Delta(t)| \leq \mu + \frac{\nu}{\Gamma(\alpha)}  \int_0^t (t-s)^{\alpha -1} |\Delta(s)|\,ds, \qquad t \in [0,T].
	$$
	Then
	$$
	|\Delta(t)| \leq \mu E_{\alpha}(\nu t ^{\alpha}), \qquad t \in [0,T].
	$$
\end{lemma}
\begin{theorem}\label{ContinuityDependence}
	Assume  that the vector field $g$ satisfies the global Lipschitz condition  in Assumption H1. Then for any $T>0$ and for each $f$ $\in$ $C([0,T],\mathbb{R}^d)$ the unique
	solution $x(t,f)$ of the integral equation  \eqref{SVIE}  depends continuously on $f$ in the supremum norm.
\end{theorem}
\begin{proof}
	
	Let $x_f$,  $y_h$ $\in$  $\in$ $C([0,T],\mathbb{R}^d)$  be the unique solutions of  \eqref{SVIE}  corresponding to the inputs $f$, $h$ $\in$ $C([0,T],\mathbb{R}^d)$. Then,
	\[
	x_f(t)-y_g(t)=f(t)-g(t)+\int_{0}^t a(t,s)(g(x(s))-g(y(s)))ds.
	\]
	Thus,
	\[
	\|x_f(t)-y_g(t)\|
	\leq
	\|f(t)-g(t)\|+ L \int_{0}^t a(t,s)\|x(s)-y(s)\|\;ds.
	\]
	The fractional Gronwall Lemma \ref{fGron} then gives
	$$
	\left\|x_f(t)-y_h(t)\right\| \leq  \left\|f(t)-h(t)\right\|  E_{\alpha}(L t^{\alpha}), \qquad 0\leq t \leq T,
	$$
	so,  in the supremum  norm on $C([0,T],\mathbb{R}^d)$,
	$$
	\left\|x_f-y_h\right\|_{\infty} \leq  \left\|f-h \right\|_{\infty}   \sup_{0\leq t\leq T} E_{\alpha}(L t^{\alpha})
	\leq  \left\|f-h \right\|_{\infty}    E_{\alpha}(L T^{\alpha}).
	$$
	The proof is complete.
\end{proof}

\section{Semi-group formulation}
Let  $\mathfrak{C}$  be the Banach space of continuous functions $f:\R^+\rightarrow \R^d$   with the topology uniform convergence on compact subsets. This topology is induced by the  metric
\[
\rho(f,h):=\sum_{n=1}^{\infty}\frac{1}{2^n} \rho_n(f,h),
\]
where
\[
\rho_n(f,h):=\frac{\sup_{t\in[0,n]}\|f(t)-h(t)\|}{1+\sup_{t\in[0,n]}\|f(t)-h(t)\|}.
\]
We follow  Chapter XI,  pages 178-179,  in Sell \cite{sell} closely,  simplifying it to this ``autonomous" case, and show that the  singular  Volterra integral equation \eqref{SVIE} generates
an autonomous semi-dynamical system on the space $\mathfrak{C}$.

\smallskip

Given  $f$ $\in$ $\mathfrak{C}$ define the operator $T_\tau$ $:$ $\mathfrak{C}$ $\rightarrow$ $\mathfrak{C}$ for each $\tau$ $\in$ $\R^+$ by
\begin{equation}\label{Semigroup}
(T_{\tau} f)(\theta) =f(\tau + \theta)+   \int_0^{\tau} a(\tau+\theta,s)  g(x_f(s))\;ds,  \qquad \theta \in \R^+,
\end{equation}
where $x_f$ is a solution of the  singular  Volterra integral equation \eqref{SVIE} for this $f$, i.e.,
$$
x_f(t)=f(t)+   \int_0^t a(t,s)  g(x_f(s))\;ds.
$$

\begin{theorem}
	Suppose that  the vector field  $g$ is globally Lipschitz continuous.  The  integral equation  \eqref{SVIE}  generalisation of the  autonomous  Caputo fractional differential equation \eqref{acfde}
	generates a semi-group of continuous operators   $\{T_\tau, \tau \in \R^+\}$ on the space $\mathfrak{C}$.
\end{theorem}

\begin{proof}
	We first show that $T_{\tau}:\mathcal C\rightarrow \mathcal C$ is continuous. Let $f,h\in \mathcal C$. Then, by \eqref{Semigroup}
	\begin{eqnarray*}
		\|T_{\tau}f(\theta)-T_{\tau}g(\theta)\|
		&\leq&
		\|f(\tau+\theta)-h(\tau+\theta)\|\\
		&&+ L \sup_{s\in[0,\tau]}\|x_f(s)-x_h(s)\|\int_0^{\tau} a(\tau+\theta,s) \;ds,
	\end{eqnarray*}
	where $L$ is the Lipschitz constant of $g$. A direct computation yields that
	\[
	\int_0^{\tau} a(\tau+\theta,s) \;ds=\frac{1}{\Gamma(\alpha)}\int_0^{\tau} (\tau+\theta-s)^{\alpha-1}\;ds=\frac{1}{\alpha\,\Gamma(\alpha)}\left((\tau+\theta)^{\alpha}-\theta^{\alpha}\right).
	\]
	Now, choose and fix $k\in\N$ with $k\geq \tau$. Then,
	\[
	\sup_{\theta\in[0,n]}\|T_{\tau}f(\theta)-T_{\tau}g(\theta)\|
	\leq
	\sup_{t\in [0,k+n]}\|f(t)-g(t)\|
	+\frac{L(k+n)^{\alpha}}{\alpha\Gamma(\alpha)} \sup_{s\in[0,\tau]}\|x_f(s)-x_h(s)\|.
	\]
	Using inequality $\frac{x}{1+x}\leq \frac{y}{1+y}+z$ provided that $x,y,z$ are non-negative and $x\leq y+z$ yields that
	\[
	\rho_n(T_{\tau}f,T_{\tau}g)\leq \rho_{n+k} (f,g)+\frac{L(k+n)^{\alpha}}{\alpha\Gamma(\alpha)} \sup_{s\in[0,\tau]}\|x_f(s)-x_h(s)\|.
	\]
	Thus,
	\[
	\rho(T_{\tau}f,T_{\tau}g)\leq 2^k \rho(f,g)+ \frac{Lc}{\alpha\Gamma(\alpha)} \sup_{s\in[0,\tau]}\|x_f(s)-x_h(s)\|,
	\]
	where $c:=\sum_{n=1}^{\infty}\frac{(k+n)^{\alpha}}{2^n}$. By virtue of Theorem \ref{ContinuityDependence}, $\sup_{s\in[0,\tau]}\|x_f(s)-x_h(s)\|\to 0$ as $\rho(f,g)\to 0$. Consequently, $T_{\tau}$ is continuous.
	
	\smallskip
	
	To complete the proof, we show that $\{T_{\tau}:\tau\in\R^+\}$ forms a semi-group. Note that
	$$
	x_f(t)=f(t)+   \int_0^t a(t,s)  g(x_f(s))\;ds.
	$$
	Then
	\begin{eqnarray*}
		x_f(t+\tau) &  =  & f(t+\tau)+   \int_0^{t+\tau} a(t+\tau,s)  g(x_f(s))\;ds
		\\[1.2ex]
		&  =  & f(t+\tau)+   \left( \int_0^{\tau} + \int_{\tau}^{t+\tau} \right) a(t+\tau,s)  g(x_f(s))\;ds
		\\[1.2ex]
		&  =  & (T_{\tau} f)(t) + \int_{\tau}^{t+\tau} a(t+\tau,s)  g(x_f(s))\;ds
		\\[1.2ex]
		&  =  & (T_{\tau} f)(t) +\int_{0}^{t }   a(t+\tau,r+ \tau)  g(x_f(r + \tau))\;dr, \qquad  (r = s -\tau),
		\\[1.2ex]
		&  =  & (T_{\tau} f)(t) +\int_{0}^{t }   a(t,r)  g(x_f(r + \tau))\;dr .
	\end{eqnarray*}
	Hence by the existence and uniqueness of solutions $x_f(t+\tau)=\psi(t)$, where $\psi(t)$ is a solution of
	$$
	\psi(t)= (T_{\tau} f)(t)  +   \int_0^t a(t,s)  g(\psi(s))\;ds.
	$$
	We also have
	\begin{eqnarray*}
		\left( T_{\sigma}\left(T_{\tau} f)\right)\right)(\theta)  &  =  & (T_{\tau} f)(\sigma+\theta) +\int_{0}^{\sigma}   a(\sigma+\theta,s)  g(\psi(s))\;ds
		\\[1.2ex]
		&  =  &  f(\tau+\sigma+\theta) +  \int_0^{\tau} a(\tau+\sigma +\theta,s)  g(x_f(s))\;ds
		\\[1.2ex]
		&    &   \qquad  \qquad   +\int_{0}^{\sigma}   a(\sigma+\theta,s)  g(\psi(s))\;ds
		\\[1.2ex]
		&  =  &  f(\tau+\sigma+\theta) +  \int_0^{\tau} a(\tau+\sigma +\theta,s)  g(x_f(s))\;ds
		\\[1.2ex]
		&    &   \qquad  \qquad  + \int_{\tau}^{\tau+ \sigma}   a(\sigma+\theta,r-\tau)  g(\psi(r-\tau)) \;dr, \qquad  (r = s +\tau).
	\end{eqnarray*}
	Since $a(\sigma+\theta,r-\tau)$ $=$ $a(\tau+\sigma+\theta,r)$ and $\psi(r-\tau)=x_f(r)$ it follows that
	\begin{eqnarray*}
		\left( T_{\sigma}\left(T_{\tau} f)\right)\right)(\theta)  &  =  &f(\tau+\sigma+\theta) +  \int_0^{\tau} a(\tau+\sigma +\theta,s)  g(x_f(s))\;ds
		\\[1.2ex]
		&     &   \qquad  \qquad  + \int_{\tau}^{ \tau+ \sigma}   a(\tau+ \sigma+\theta,r)  g(x_f(r)) \;dr.
	\end{eqnarray*}
	This gives
	$$
	\left( T_{\sigma}\left(T_{\tau} f\right)\right)(\theta) = f(\tau+\sigma+\theta) +  \int_0^{\tau+\sigma} a(\tau+\sigma +\theta,s)  g(x_f(s))\;ds
	$$
	On  the other hand from the definition of the operator as in \eqref{Semigroup}
	$$
	\left(T_{\sigma+\tau} f\right)(\theta) = f(\tau+\sigma+\theta) +  \int_0^{\tau+\sigma} a(\tau+\sigma +\theta,s)  g(x_f(s))\;ds
	$$
	This means that
	$$
	\left(T_{\sigma+\tau} f\right)(\theta) =  \left( T_{\sigma}\left(T_{\tau} f\right)\right)(\theta), \qquad \forall \tau, \theta, \sigma  \geq 0, \,\, f \in \mathfrak{C},
	$$
	that is
	$$
	T_{\sigma+\tau} f =   T_{\sigma}\left(T_{\tau}\right) f,\qquad \forall  \tau,  \sigma  \geq 0, \, \,\, f \in \mathfrak{C}.
	$$
	
\end{proof}

\smallskip

\begin{remark} As in \cite{sell},  we say that a vector field $g$ in the integral  equation \eqref{SVIE}  is \emph{admissible} if it has a globally defined unique solution for each  $f$ $\in$
	$\mathfrak{C}$ with continuity in initial data. This holds if $g$ is globally Lipschitz continuous as above, but  weaker assumptions are also possible.  The theorem below  also holds for such
	admissible vector fields.
\end{remark}

\begin{remark}
	The above results   also hold for autonomous  Caputo fractional differential equations with a substantial time derivative, i.e., of the form
	$$
	x(t)=x_0+\frac{1}{\Gamma(\alpha)}  \int_0^t (t-s)^{\alpha -1}e^{-\beta (t-s)} g(x(s))\;ds,
	$$
	where $\beta$ $>$ $0$.  This can be seen by replacing $a(t,s)$ by
	$$
	\tilde{a}(t,s) :=  \frac{1}{\Gamma(\alpha)}  (t-s)^{\alpha -1}e^{-\beta (t-s)},  \qquad 0\leq s < t.
	$$
	Note that $0$ $<$ $\tilde{a}(t,s)$ $\leq$ $a(t,s)$.
\end{remark}

\section{Attractors of the Caputo semi-dynamical system}

The theory of autonomous semi-dynamical systems \cite{KR} can be applied to the Caputo semi-group defined above.

The solution $x(t,x_0)$  of the  autonomous  Caputo FDE  \eqref{acfde} corresponds to a  constant function $f_0(t)$ $\equiv$ $x_0$ and
$$
x(t,x_0)  \equiv (T_t f_0)(0).
$$
Thus, when the semi-group    $\{T_\tau, \tau \in \R^+\}$ has an attractor  $\mathfrak{A}$ $\subset$   $\mathfrak{C}$,   then   an  omega limit point $x$ $\in$ $\mathbb{R}^d$  of   trajectories of the Caputo FDE  satisfies $x$ $=$ $f(0)$ for some   function $f$  $\in$   $\mathfrak{A}$.

In particular, if $g(x^*)$ $=$ $0$, then $f^*$  $\in$   $\mathfrak{A}$ for the constant function $f^*(t)$   $\equiv$ $x^*$, i.e.,  $x^*$ is a steady  state solution of the system. But   there may be   functions $f^*$  $\in$   $\mathfrak{A}$  that are not constant functions, so the strict inclusion,  $\Omega$ $\subsetneq$  $\mathfrak{A}(0)$ usually holds, where $\Omega$
is the union of all the above   omega limits points.

\section{Non-autonomous Caputo  FDE: skew-product  flow}

The above result can be extended to  the nonautonomous  case with a time dependent vector field $g(t,x)$. Then, again following Sell \cite{sell}, we can show that a non-autonomous Caputo  fractional differential equation generates a skew-product flow. We just sketch the details here.

In particular, we now use the  integral equation
\begin{equation}\label{NAIE}
x(t)=x_0+\frac{1}{\Gamma(\alpha)}
\int_0^t (t-s)^{\alpha -1} g(s, x(s))\;ds.
\end{equation}
We define the shift mappings
$$
g_\tau(\cdot,x) = g(\tau+ \cdot, x)
$$
and  (to match Sell's notation in  \cite{sell})
$$
a_\tau(\cdot,\cdot) = a(\tau+ \cdot, \tau+ \cdot),
$$

\medskip

Then, following  Sell  \cite{sell}, we  define
$$
(T_{\tau} (f,g))(\theta) =f(\tau + \theta)+   \int_0^{\tau} a_\tau(t+\theta,s)  g_\tau(\phi(s))\;ds,
$$
so in our case  we have in fact
$$
(T_{\tau} (f,g))(\theta) =f(\tau + \theta)+   \int_0^{\tau} a(t+\theta,s)  g_\tau(\phi(s))\;ds,  \qquad \theta \in \R^+.
$$
(Since  $a_\tau(t,s) = a(t, s)$ is  our case, it need not be considered as an independent  variable here).

\medskip

In the autonomous case  $a$ and $g$ were  fixed  functions, so they appeared  just parameters in  the  operators $T_\tau$.
Now, both  $f$ and $g$ can vary in time,  so they are the independent variables that  determine  the  operators $T_\tau$.

\medskip

Let $\mathfrak{G}$ be an appropriate space of admissible functions  $g$ $:$ $\R^+ \times \R^d$  $\rightarrow$ $\R^d$, see Sell  \cite{sell} for some examples of such spaces. We can introduce a semi-group  $\theta_\tau$ $:$ $\mathfrak{G}$  $\rightarrow$ $\mathfrak{G}$ defined by the shift $\theta_\tau g$ $:=$ $g_\tau$ as our ``driving system".  Then we obtain  a \emph{skew-product flow}
$$
\Pi : \R^+ \times \mathfrak{C} \times \mathfrak{G}  \rightarrow \mathfrak{C} \times \mathfrak{G}
$$
defined by
$$
\Pi(\tau,f,g) := \left(T_{\tau} (f,g), g_\tau \right).
$$

\smallskip

The proof   is  similar to that above with a bit more complicated  notation. It is  exactly as in Sell \cite{sell}, pages 178-179. Essentially, here the operator $T_\tau(f,g)$ $:$
$\mathfrak{C} \times \mathfrak{G}$  $\rightarrow$ $\mathfrak{C}$ for each $\tau$ $\in$ $\R^+$ satisfies a cocycle property with respect to the driving system $\theta$, \cite{KR}.


%
%

\begin{thebibliography}{99}
	
	\bibitem{cong} N.D. Cong and H.T. Tuan,  \emph{Generation of nonlocal dynamical systems by fractional differential equations}, J.  Integral Eqns.\& Applns.,
	{\bf 29} (2017), 585--608.
	
	\bibitem{diethelm}  K. Diethelm, \emph{The Analysis of Fractional Differential Equations}. Springer Lecture Notes in Mathematics, vol. 2004,
	Springer, Heidelberg, 2010.
	
	\bibitem{doan} Doan Thai Son,  Thi Huong Phan, Peter E. Kloeden and
	The Tuan Hoang, Asymptotic separation between solutions of Caputo fractional stochastic differential equations,
	{\em J. Stochastic Anal \& Applns.} {\bf 39} (2018), 654--664.
	
	\bibitem{KR} P. E. Kloeden and M. Rasmussen, {\it Nonautonomous dynamical systems},  American Mathematical Society, Providence (2011).
	
	\bibitem{miller} R.K. Miller,  \emph{Nonlinear  Volterra Integral Equations}. W.A. Benjamin, Menlo Park, 1971.
	
	\bibitem{MS1} R.K. Miller and G.R. Sell,  \emph{ Volterra Integral Equations and Topological Dynamics}. Memoir Amer. Math. Soc. vol. 102, 1970.
	
	
	
	\bibitem{MS2} R.K. Miller and G.R. Sell,   Existence, uniqueness and continuity of solutions
	of integral equations. An Addendum, \emph{Annali di Matematica Pura ed Applicata},  {\bf 87} (1970),  281--286.
	
	\bibitem{sell}  G.R. Sell.
	{\em Topological {D}ynamics and {O}rdinary {D}ifferential {E}quations}.
	Van Nostrand Reinhold Mathematical Studies, London, 1971.
\end{thebibliography}
\end{document}